\documentclass[a4paper,draft]{amsart}
 \usepackage{amsmath,amssymb,amsthm,latexsym}
\begin{document}

   \begin{center}
    \LARGE \bfseries Bases for partially commutative Lie
    algebras\\[1.3em]
    \mdseries \large E.\,N.\, Poroshenko\\[0.2em]
         \itshape  Department of Algebra 
           and Mathematical Logic, \\
        Novosibirsk State Technical University,\\
        Novosibirsk, Russia\\
        (e-mail: auto\_stoper@ngs.ru)
   \end{center}
   \author{}
   \title{}
   \begin{abstract}
     In this paper, linear bases for the partially commutative Lie algebras are found.
     The method of the Gr\"{o}bner--Shirshov bases is used.
   \end{abstract}

   \maketitle

   \theoremstyle{definition}
   \newtheorem{ddd}{Definition}[section]
   \newtheorem{eee}[ddd]{Example}
   \theoremstyle{plain}
   \newtheorem{ttt}[ddd]{Theorem}
   \newtheorem{llll}[ddd]{Lemma}
   \theoremstyle{remark}
   \newtheorem*{rrr}{Remark}

   \section{Introduction and Preliminaries}

   Let
   $X$ be a finite set and let
   $G=\langle X,E\rangle$ be an undirected graph without loops
   with the set of vertices
   $X$ and the set of edges
   $E$. Since the graph
   $G$ is undirected the elements of
   $E$ are unordered pairs that we denote by
   $\{x,y\}$, where
   $x,y\in X$.
   \begin{ddd} \label{pkalgli}
      Let
      $R$  be a commutative associative ring with unit.
      A \emph{partially commutative Lie algebra over}
      $R$ is just the Lie
      $R$-algebra under the ring
      $R$ with the set of generators
      $X$ and the set of defining relations
      \begin{equation} \label{rel}
         (x_i,x_j)=0, \text{ for }\{x_i,x_j\} \in E
      \end{equation}
      (thereafter, we denote the Lie product of
      $x$ and
      $y$) by
      $(x,y)$.

      The graph
      $G$ is called a \emph{defining graph} for the corresponding algebra that we denote by
      $\mathcal{L}_{R}(G)$. If there is no ambiguity we omit the subscript and write
      $\mathcal{L}(G)$ instead of
      $\mathcal{L}_{R}(G)$.
  \end{ddd}

   So, the definition of the partially commutative Lie algebras is analogous to ones of other
   partially commutative structures such as groups, monoids etc. (see~%
 \cite{DK93}).

   Partially commutative groups (the term ``graph
   groups'' is also used) are studied very heavily nowadays (see~%
\cite{Ser89,DK92,She05,She06,DKR07,GT09,T10}, for example).
   Although, there are
   some results obtained for other partially commutative structures (see
\cite{KMNR80,BS01,DK92'}).

   For instance, it was shown in
\cite{KMNR80} that if two partially commutative associative
   algebras (they are called ``graph algebras'' in that paper) are isomorphic then so are
   their defining graphs.  Actually, in that paper, the partially
   commutative algebra corresponding to the graph
   $G\langle X,E\rangle$ is defined as an algebra with the set of generators
   $X$ and the set of defining relations
   $x_ix_j=x_jx_i$, where
   $\{x_i,x_j\}\not \in E$. It means that the associative partially commutative algebra corresponding to the graph
   $G$ has the defining graph
   $\overline{G}$ (in meaning of the terminology used nowadays), i.e. the complement of
   $G$. However, it makes no difference because, obviously,
   $G\simeq G'$ if and only if
   $\overline{G} \simeq \overline{G'}$.

   Some results are also obtained for partially commutative Lie
   algebras. So, in
\cite{DK92'}, the algorithm finding a basis for any partially
   commutative Lie algebra is given. However, this algorithm is based on a
   decomposition of the set
   $X$ by two subsets one of which is independent. For this
   reason, the final result depends essentially on the structure
   of the graph and, therefore, there are deep difficulties in applying this algorithm for
   the explicit description of the bases for partially commutative Lie algebras in general.

   The goal of this paper is to make up for a deficiency and find the bases for partially commutative
   Lie algebras explicitly. We are using totally different methods, namely the technique
   of the Gr\"{o}bner--Shirshov bases.

   For a partially commutative algebra, the basis found in this paper consists of Lyndon--Shirshov words.
   For this reason, we should remind the corresponding definitions.

   Let
   $X^*$ be the set of all associative words in the alphabet
   $X$ (including the empty word denoted by
   $1$). Let us set a linear order on
   $X$ and extend it to the linear order on
   $X^*$ by two different ways:
   \begin{enumerate}
       \item
         $u<1$ for any non-empty word
         $u$. By induction,
         $u<v$, if
         $u=x_iu'$,
         $v=x_jv'$, where
         $x_i,x_j\in X$, and either
         $x_i<x_j$ or
         $x_i=x_j$ and
         $u'<v'$. This is so called lexicographic order.
       \item
         $u\prec v$ if either
         $\ell (u)< \ell (v)$ or
         $\ell (u)= \ell (v)$ and
         $u<v$, where
         $\ell(u)$ is the length of
         $u$. This order is denoted by deg-lex.
     \end{enumerate}

   \begin{ddd}\label{alsw}
      An associative word
      $u\in X^*$ is called an \emph{associative Lyndon--Shirshov word}, if
      for an arbitrary non-empty words
      $v,w\in X^*$ such that
      $u=vw$, we have
      $wv<u$.
   \end{ddd}

   We also consider the set if all non-associative words in
   $X$ (here, we exclude the empty word from the consideration),
   i.e the set of words with all possible bracketings (let us denote this set by
   $X^{+}$). In the next definition, if
   $[u]$ is an arbitrary non-associative word then
   $u$ denotes the word obtained from
   $[u]$ by removing all brackets.
   \begin{ddd}\label{nalsw}
      \renewcommand{\theenumi}{\roman{enumi}}
      A non-associative word
      $[u]$ is called a \emph{Lyndon--Shirshov word} if
      \begin{enumerate}
         \item
           $u$ is an associative Lyndon-Shirshov word;
         \item
           if
           $[u]=([u_1],[u_2])$ then
           $[u_1]$ and
           $[u_2]$ are non-associative Lyndon--Shirshov words and
           $u_1>u_2$.
         \item
           If
           $[u_1]=([u_{11}],[u_{12}])$ then
           $u_2\geqslant u_{12}$.
      \end{enumerate}
   \end{ddd}
   Let us denote the set of all associative Lyndon--Shirshov words in
   $X$ by
   $LSA(X)$ and the set of all non-associative Lyndon--Shirshov words in
   the same alphabet by
   $LS(X)$.

   By
\cite{sh62'},
   $LS(X)$ is a linear basis of the free Lie algebra with the set in generators
   $X$ over an arbitrary field. We can easily conclude from this that this set is also a basis of the free
   Lie
   $R$-algebra. Let us denote this algebra by
   $\mathrm{Lie}_{R}(X)$.

   In
\cite{sh58}, it was shown that for any associative
   Lyndon--Shirshov word there is the unique bracketing making a non-associative
   Lyndon--Shirshov word. It means that there exists a bijection
   $[\,\cdot\,]: LSA(X)\rightarrow LS(X)$. So, from now on, we denote the
   image of
   $u\in LSA(X)$ under this bijection by
   $[u]$. Finally, for
   $[u],[v]\in LS(X)$ we write
   $[u]<[v]$
   ($[u]\prec [v]$) if
   $u<v$ (%
   $u\prec v$ respectively). In a similar manner, we understand the notations:
   $[u]\leqslant [v]$,
   $[u]\preceq [v]$,
   $[u]>[v]$,
   $[u]\succ [v]$,
   $[u]\geqslant [v]$,
   $[u]\succeq [v]$.

   Let
   $f=\sum_{i}\alpha_i [v_i]$ be a Lie polynomial that is a linear combination of
   non-associative Lyndon--Shirshov words. Denote by
   $\overline{f}$ the monomial
   $\alpha_k [v_k]$ such that
   $[v_k]\succ [v_i]$ for any
   $i\neq k$.

   Now, let us remind the notion of the composition (see~%
\cite{sh62}). Let
   $u=vdw$, where
   $u,d\in LSA(X)$. By
\cite{sh58}, the minimal non-associative subword
   $[u']$ of the word
   $[u]$ such that
   $u'$ covers
   $d$ is easily seen to be of the form
   $[dc]$, for some
   $c\in X^*$. Obviously, if
   $c\neq 1$ then
   $c=c_1c_2\dots c_m$, where
   $c_i\in LSA(X)$ for
   $i=1,2,\dots m$ and
   $c_1\leqslant c_2 \leqslant \dots \leqslant c_m$. Denote by
   $[vdw]_{d}$ the non-associative word obtained from
   $u$ by replacing the subword
   $[dc]$ by
   $((\dots(([d],[c_1]),[c_2])\dots ),[c_m])$. Let
   $f$ be a monic Lie polynomial (i.e. a polynomial whose coefficient by
   $\overline{f}$ is equal to
   $1$). Denote by
   $[vfw]_{\overline{f}}$ the Lie polynomial obtained from
   $[v\overline{f}w]_{\overline{f}}$ by replacing
   $\overline{f}$ by
   $f$.

   \begin{ddd}
      Let
      $f$ and
      $g$ be monic Lie polynomials and let
      $w\in X^*$ be such that
      $w=\overline{f}a=b\overline{g}$, where
      $a,b\in{X^*}$ and
      $\ell(\overline{f})+\ell(\overline{g})>\ell(w)$. The \emph{composition
      of intersection} of
      $f$ and
      $g$ relative to
      $w$ is defined by
      $$(f,g)_{w}=[fa]_{\overline{f}}-[bg]_{\overline{g}}.$$
   \end{ddd}

   \begin{ddd}\label{inclus}
      Let
      $f$ and
      $g$ be monic Lie polynomials and let
      $w\in X^*$ be such that
      $w=\overline{f}=a\overline{g}b$, where
      $a,b\in{X^*}$. The \emph{composition of inclusion} of
      $f$ and
      $g$ relative to
      $w$ is defined as follows:
      $$(f,g)_{w,a}=f-[agb]_{\overline{g}}.$$
   \end{ddd}

   For simplicity, we denote the composition of inclusion also by
   $(f,g)_{w}$. Let us note that
   $(f,g)_w\in \mathrm{Id}(f,g)$ and
   $\overline{(f,g)}_w\prec w$ (here
   $\mathrm{Id}(f,g)$ is an ideal of
   $\mathrm{Lie}_R(X)$ generated by
   $f$ and
   $g$). Analogously, if
   $S$ is a set of Lie polynomials in
   $X$, then
   $\mathrm{Id}(S)$ denotes the ideal of
   $\mathrm{Lie}_{R}(X)$ generated by
   $S$.

   \begin{ddd}
      Given a set
      $S$ of monic Lie polynomials, the composition
      $(f,g)_w$ of
      $f,g\in S$ is \emph{trivial  relative to}
      $S$ if
      $(f,g)_w=\sum_i \alpha_i[a_is_ib_i]$, where
      $\alpha_i\in R$,
      $a_i,b_i \in X^*$,
      $s_i\in S$ и
      $a_i\overline{s}_ib_i\prec w$.
   \end{ddd}

   \begin{ddd}
     Let
     $S$ be the set of monic Lie polynomials.
     $S$ is a \emph{Gr\"{o}bner--Shirshov basis} of
     $\mathrm{Lie}_{R}(X)/\mathrm{Id}(S)$ if every composition of any two elements
     is trivial relative to
     $S$.
   \end{ddd}

   \begin{ddd}
     Let
     $[u],[v]\in LS(X)$. We say that
     $[v]$ is a \emph{subword} of
     $[u]$ if an associative word
     $v$ is a subword of an associative word
     $u$, i.e. if there exist
     $a,b\in X^*$ such that
     $avb=u$.
   \end{ddd}

   \begin{ddd}
     The word
     $[u]\in LS(X)$ is called
     $S$-\emph{reduced} if
     $u\neq a\overline{s}b$ for any
     $s\in S$ and
     $a,b\in X^*$.
   \end{ddd}
   The main result we need is the following one:
   \begin{llll}[CD-Lemma]
     {\rm (}see
\cite{Bo72,BFKK00}{\rm )}
     $S$ is a Gr\"{o}bner--Shirshov basis if and only if the set of all Lyndon--Shirshov
     $S$-reduced words is a linear basis for
     $\mathrm{Lie}_{R}(X)/\mathrm{Id}(S)$.
   \end{llll}

   Finally, let us remind a couple properties of the Lyndon---Shirshov words that we are going
   to use in this paper.

   \begin{llll} \emph{(\cite{BC08}, Lemma~2.12)} \label{LSAL}
      If
     $u,v\in LSA(X)$ and
     $u>v$, then
     $uv\in LSA(X)$.
   \end{llll}

   \begin{llll}\label{less}
     Let
     $[u],[v]\in LS(X)$ and
     $[u]>[v]$. In
     $\mathrm{Lie}_{R}(X)$, we can write
     $([u],[v])=[uv]+\sum_{i}[w_i]$, where
     $[w_i]\in LS(X)$ and
     $w_i<uv$.
   \end{llll}
   \begin{proof}
    It follows clearly from
 \cite{sh58} (Lemmas~2 and 3) and from Lemma~%
 \ref{LSAL}.
    \end{proof}

   \section{A basis of
   $\mathcal{L}(G)$}

   To find a linear basis of
   $\mathcal{L}(G)$, let us find its Gr\"{o}bner--Shirshov basis.

   Let
   $G=\langle X,E \rangle$. Consider the free Lie
   $R$-algebra
   $\mathrm{Lie}_{R}(X)$ whose set of generators coincides with the set of the vertices of
   $G$. Let us denote by
   $S(G)$ the set of Lyndon--Shirshov words
   $[u]$ such that
   $[u]=([\tilde{u}],b)$, where
   $b\in X$,
   $[\tilde{u}]$ does not contain
   $b$, and
   $\{b,y\}\in E$ for any letter
   $y$ appearing in
   $[\tilde{u}]$.

   Note that if
   $[u]\in S(G)$ then its last letter is the second largest one in this word. Indeed, condition (ii)
   of Definition~%
\ref{nalsw} implies, the first letter of
   $[\tilde{u}]$ is larger then
   $b$ (the last letter of
   $[u]$). We are left to show that all other letters of
   $[\tilde{u}]$ are less than
   $b$. If
   $\ell(\tilde{u})=1$ then the statement is obvious. Otherwise,
   $[\tilde{u}]=([u_1],[u_2])$, where
   $[u_1],[u_2]\in LS(X)$.  If there is another letter of
   $[\tilde{u}]$ which is not less than
   $b$ then by condition (ii) of Definition~%
\ref{nalsw} we have the first letter of
   $[u_2]$ is greater than
   $b$ (by the definition of
   $S(G)$ it cannot be equal to
   $b$). But in this case,
   $[u]\not \in LS(X)$ by condition (iii) of Definition
\ref{nalsw}.

   The following statement holds:

   \begin{llll} \label{id}
      $S(G)\subseteq \mathrm{I}(G)$, where
      $\mathrm{I}(G)$ is the ideal of
      $\mathrm{Lie}_{R}(X)$ generated by the set of relations (%
\ref{rel}).
   \end{llll}

   \begin{proof}
      Since the Lie product is a derivation (it is the consequence of the Jacobi identity),
      $([\tilde{u}],b)$ is equal to a sum of
      $\ell(\tilde{u})$ terms each of which contains a product of the form
      $(x,b)$, where
      $x\in X$ is a letter appearing in
      $u$. This product is in
      $\mathrm{I}(G)$ by
(\ref{rel}).

   \end{proof}

   \begin{ddd}
     Given an associative word
     $d\in LSA(X)$, the word
     $[u]\in LS(X)$ is called
     $d$-\emph{decomposable} if
     $[u]=([v],[w])$, where
     $[w]>[d]$ and
     $d$-\emph{indecomposable}, otherwise.
   \end{ddd}

   Let us take
   $d\in LSA(X)$ and consider a word
   $[u]\in LS(X)$ such that
   $[u]>[d]$. If
   $[u]=([v],[w])$  is
   $d$-decomposable then consider the words
   $[v]$ and
   $[w]$. Each of which is also either
   $d$-decomposable or a product of two words each of them is greater than
   $[d]$ and so on. Since each word has a finite number of letters, any word can be represented as a product of
   $d$-indecomposable words
   $[u_1],[u_2],\dots,[u_k]$ (with some bracketing among them).

   \begin{ddd}
     Let
     $[u],[d]\in LS(X)$. The described above decomposition of
     $[u]$ as a product of
     $d$-indecomposable words
     $[u_1],[u_2],\dots,[u_k]$ is called
     $d$-\emph{decomposition} and denoted by
     $([u_1],[u_2],\dots ,[u_k])_d$ or
     $[u]_d$.
   \end{ddd}

   \begin{ddd}
    Let us define the notion of a \emph{pattern} inductively:
     \begin{enumerate}
       \item \label{p1}
         The symbol
         $*$ is a pattern;
       \item \label{p2}
         If
         $p$ and
         $q$ are patterns then
         $(p,q)$ is also a pattern;
       \item
         There are no patterns except described in
(\ref{p1})-(\ref{p2}).
     \end{enumerate}
   \end{ddd}

   \begin{ddd}
     Let
     $[u],[d]\in LS(X)$. Let us define
     a \emph{
     $d$-pattern} of
     $[u]$ as an object obtained by replacing all
     $d$-indecomposable subwords of
     $[u]$ by
     ``$*$''. Denote the
     $d$-pattern of
     $[u]$ by
     $p(u,d)$.
   \end{ddd}

   \begin{eee}
     Let
     $X=\{x_1,x_2,\dots , x_6\}$. Consider the order on it defined as follows:
     $(x_i<x_j) \Leftrightarrow (i<j)$. If
     $[u]= \Bigl(\bigl((x_6,x_3),(x_5,x_1)\bigr),\bigl(((x_6,x_1),x_3),(x_5,x_2)\bigr)\Bigr)$
     and
     $[d]=(x_4,x_2)$ then the set of
     $(x_4,x_2)$-indecomposable subwords of
     $[u]$ consists of
     $(x_6,x_3)$,
     $(x_5,x_1)$,
     $((x_6 x_1) x_3)$, and
     $(x_5 x_2)$. Therefore,
     $(x_4,x_2)$-pattern of
     $[u]$ is
     $((*\,,*)(*\,,*))$.
   \end{eee}

   \begin{rrr}
     For any
     $[u]\in LS(X)$,
     $d$-pattern shows the bracketing among its
     $d$-indecomposable multipliers.
   \end{rrr}

   Note that any
   $d$-pattern is obviously a pattern.

   Let
   $f_1,\dots ,f_k$ be Lie polynomials and let
   $p$ be a pattern containing
   $k$ symbols
   ``$*$''. We denote by
   $(f_1,\dots, f_k)_p$ the Lie polynomial obtaining by simultaneous replacing
   $i$th symbol
   ``$*$'' by
   $f_i$ (counting from left to right) for all
   $i=1,2,\dots , k$ and by transformations by the distributive law after that
   (we do not use the anticommutative law and the Jacobi identity).

   \begin{llll}\label{drazl}
     Let
     $[u]$ and
     $[d]$ be Lyndon--Shirshov words such that
     $[u]>[d]$ and all letters of
     $u$ which are greater than the the greatest letter of
     $d$ appear in
     $u$ exactly once and the greatest letter of
     $d$ is not contained in
     $u$. Let
     $[u]=([u_1],[u_2],\dots , [u_k])_d$ be a
     $d$-composition of
     $u$ and let
     $u'$ be a word obtaining from
     $u$ by replacing
     $u_r$ by
     $u_rd$ for some
     $r$. Then
     $([u_r],[d])\in{LS}(X)$,
     $[u']=
        ([u_1],[u_2],\dots ,([u_r],[d]), \dots ,[u_k])_{p(u,d)}$,
     i.e.
     $p(u,d)=p(u',d)$.
   \end{llll}
   \begin{proof}
     The first statement is obvious since
     $[u_r]$ and
     $[d]$ are Lyndon--Shirshov words,
     $[u_r]>[d]$ and if
     $[u_r]=([v],[w])$ then
     $[w]\leqslant [d]$ because
     $[u_r]$ is
     $d$-indecomposable.

     Next, obviously
     $([u_r],[d])$ is
     $d$-indecomposable. We are left to show that
     $d$-patterns of
     $[u]$ and
     $[u']$ coincide.

     Let us proceed by induction by the number of multipliers in
     $d$-decompositions (denote this number by
     $k$). For
     $k=1$, the assertion is obvious. Let
     $k\geqslant 2$. In this case,
     $[u]=([v],[w])$, where
     $[v]=([u_1],\dots , [u_l])_d$,
     $[w]=([u_{l+1}],\dots , [u_k])_d$. Since all the letters of
     $[u]$ which are greater than the greatest letter of
     $d$, appear in
     $[u]$ exactly once, the first letter of
     $[u_1]$ is the greatest letter of
     $[u]$ and the first letter of
     $[u_{l+1}]$ is the second greatest letter of
     $[u]$.

     Consider the word
     $[u']$. Since
     $[u_i]>[d]$ for all
     $i$  and since
     $[u]$ does not contain the greatest letter of
     $[d]$ the greatest letter of
     $[d]$ is less than the greatest letters of
     $[u_i]$ (%
     $i=1,2,\dots ,k$). Consequently, the second greatest letter
     of
     $[u']$ is also the first letter of
     $[u_{l+1}]$. So,
     $[u']=([v'],[w'])$, where
     $v'$ и
     $w'$ are words obtained from
     $v$ and
     $w$ by replacing
     $u_r$ by
     $u_rd$ (if
     $u_r$ is not a subword of
     $v$ (respectively not a subword of
     $w$), then we suppose
     $v'=v$ (%
     $w'=w$ respectively)). By induction hypothesis,
     $p(v,d)=p(v',d)$ and
     $p(w,d)=p(w',d)$. Therefore,
     $p(u,d)=(p(v,d),p(w,d))=(p(v',d),p(w',d))=p(u',d)$.
  \end{proof}

  \begin{llll}
     Let
     $[u]$ and
     $[d]$ be Lyndon--Shirshov words such that all the letters of
     $u$, which are greater than the greatest letter of
     $d$, appear in
     $u$ exactly once and the greatest letter of
     $d$ is not contained in
     $u$. Let
     $[u]=([u_1],[u_2],\dots , [u_k])_y$, where
     $y$ is a letter not contained in
     $u$ and such that
     $y>d$. Finally, let the word
     $u'$ be obtained from
     $u$ by replacing
     $u_r$ by
     $u_ry$. Then the
     $d$-pattern of
     $[u']$ is obtained from
     $d$-pattern of
     $[u]$ by replacing its part
     $p(u_r,d)$, corresponding to
     $[u_r]$ by
     $(p(u_r,d),*)$. In other words, if
     $[u]=([v_1],\dots,[v_t],\dots, [v_s])_d$, where
     $[v_t]$ is the last multiplier in the
     $d$-decomposition of
     $[u_r]$, then
     $[u']=([v_1,\dots,[v_t],y,[v_{t+1}]\dots,[v_s])_d$.
   \end{llll}
   \begin{proof}
     It is obvious that if
     $[u]\in LS(X)$ and
     $y>d$ then
     $$[u]_d=([u_{1,1}],\dots, [u_{1,l_1}],\dots,[u_{k,1}],\dots,[u_{k,l_k}])_d,$$
     where
     $([u_{i1}],\dots, [u_{il_1}])_d=[u_i]_d$ for
     $i=1,2,\dots,k$. In other words, it means that the
     $d$-pattern of
     $[u]$ is obtained from the
     $y$-pattern of this word by simultaneous  replacing
     the $i$th symbol
     ``$*$'' by
     $d$-pattern of the corresponding word
     $[u_i]$.

     By Lemma~%
\ref{drazl},
     $y$-patterns of
     $[u]$ and
     $[u']$ coincide. Consequently, we only need to compare the
     $d$-patterns of
     $[u_r]$ and
     $[u_ry]$. We have
     $[u_ry]=([u_r],y)$ since
     $y>d$ and
     $u_r$ is
     $y$-indecomposable. So, the assertion follows.
   \end{proof}

   Now, we can start computing compositions. Let
   $[u],[v]\in S(G)$. Suppose that
   $[u]>[v]$. It follows from the structure of the words of
   $S(G)$ that a composition of intersection for these two words in
   $S(G)$ exists if and only if the first letter of
   $[v]$ is equal to the last letter of
   $[u]$. Moreover, in this case, there exists the unique composition of
   $[u]$ and
   $[v]$. This is the composition relative to
   $w=\tilde{u}v$, where
   $\tilde{u}$ is the associative word obtained by deleting the last letter in
   $u$.

   Given
   $[u]\in LS(X)$,
   $d\in LSA(X)$, we introduce the map
   $\partial_d$:
   $$\partial_d[u]=\sum_{i=1}^k ([u_1],\dots, ([u_i],d),\dots [u_k])_d,$$
   where
   $[u]=([u_1],\dots,[u_k])_d$. In other words, if the first
   letters of
   $[u_1]\dots [u_k]$ are distinct then applying this map to
   $[u]$ gives us the representation of
   $([u],[d])$ as a sum of Lyndon--Shirshov words. It is clear that
   if
   $([v],[w])$ is a Lyndon--Shirshov word then
   $$\partial_d ([v],[w])=\bigl((\partial_d [v]),[w]\bigr)+\bigl([v],(\partial_d [w])\bigr),$$
   i.e.
   $\partial_d$ is a derivation. By induction, for Lyndon--Shirshov words
   $[u]=([u_1],\dots,[u_k])_d$ and
   $[v]<[d]$ we obtain:
   $$\partial_v [u]=\sum_{i=1}^k([u_1], \dots, \partial_v [u_i], \dots, [u_k])_{p(u,d)}.$$

   First of all, let us consider the case
   $v=by_1y_2\dots y_k c$, where
   $b,c,y_1,\dots, y_k\in X$ such that
   $b>c>y_k\geqslant \dots\geqslant y_1$. Let
   $[u]=([\tilde{u}],b)$ be such that each letter of
   $[\tilde{u}]$ that is greater than
   $y_1$ appears in
   $[u]$ exactly once, and
   $[\tilde{u}]$ does not contain letters of
   $[v]$. Let
   $\widehat{S}(G,v)$ be the set of such words
   $[u]$.

   \begin{llll}\label{main}
     Let
     $v=by_1y_2\dots y_k c$, where
     $b,c,y_1,\dots, y_k\in X$ such that
     $b>c>y_k\geqslant \dots \geqslant y_1$ and
     $[u]\in \widehat{S}(G,v)$. Moreover, let
     $[u]=([\tilde{u}],b)$. Then the composition of
     $[u]$ and
     $[v]$ is trivial relative to
     $S(G)$.
   \end{llll}
   \begin{proof}
     Let us note that the vertex
     $b$ is adjacent in
     $G$ to the vertices corresponding to the letters of
     $\tilde{u}$, and
     $c$ is adjacent to
     $b,y_1,\dots, y_k$. We are going to compute the composition
     $([u],[v])_{[\tilde{u}v]}$.

     For an arbitrary set
     $A=\{v_1,v_2,\dots,v_s\} \subseteq LSA(X)$ such that
     $v_1\leqslant v_2 \leqslant \dots \leqslant v_s$, we use the following notation:
     $$\partial_{A} [u]=\partial_{v_s}\partial_{v_{s-1}}\dots \partial_{v_1}[u],$$

     Let
     $\Omega_m(A)$ be the set of all ordered decompositions of
     $A$ by
     $m$ subsets (some of them may be empty), i.e.
     $$\Omega_m(A)=\left\{(A_1,\dots, A_m)\,\left|\,
       \bigcup_{t=1}^m A_t=A;\ A_i\cap A_j=\varnothing, \text{ if }i\neq j\right.\right\}.$$
     Henceforth, we will denote the set
     $\{y_1,\dots,y_k\}$ by
     $\Delta$ and the set
     $\{y_1,\dots, y_k,c\}$ by
     $\Delta'$. Let
     $[\tilde{u}]=([u_1],\dots,[u_r])_c$. We obtain
     \begin{equation}
     \begin{split}\label{sum}
       ([u],&[v])_{[\tilde{u}v]}  =
            \partial_{\Delta'} [u]-([\tilde{u}],\partial_{\Delta'} b)\\
           & =  \sum_{(\Delta_1,\dots, \Delta_{r+1})\in \Omega_{r+1}(\Delta')}
           \bigl((\partial_{\Delta_1}[u_1],\dots,
           \partial_{\Delta_r}[u_r])_{p(\tilde{u},c)},\partial_{\Delta_{r+1}}b\bigr)
           -([\tilde{u}],\partial_{\Delta'} b)\\
           & = \sum_{(\Delta_1,\dots, \Delta_{r+1})\in \Omega_{r+1}(\Delta)}
           \bigl((\partial_{\Delta_1}[u_1],\dots,
           \partial_{\Delta_r}[u_r])_{p(\tilde{u},c)},\partial_c
           \partial_{\Delta_{r+1}}b\bigr)\\
           & + \sum_{i=1}^{r}\
            \sum_{(\Delta_1,\dots, \Delta_{r+1})\in \Omega_{r+1}(\Delta)}
           \bigl((\partial_{\Delta_{1}}[u_1],\dots,\partial_c\partial_{\Delta_{i}}[u_i],\dots,\partial_{\Delta{r}}[u_r])_{p(\tilde{u},c)}
           \partial_{\Delta_{r+1}}b\bigr)\\
           &-([\tilde{u}],\partial_{\Delta'} b)
     \end{split}
     \end{equation}

     In the first sum, all summands can be rewritten as sums of Lyndon--Shirshov words
     by the distributive law. Moreover, if
     $\Delta_{r+1}\neq \Delta$, then the corresponding summand
     $\bigl((\partial_{\Delta_1}[u_1],\dots,
           \partial_{\Delta_r}[u_r])_{p(\tilde{u},c)},\partial_c
           \partial_{\Delta_{r+1}}b\bigr)$ is a sum of words less than
     $[\tilde{u}v]$. There is a multiplier of the form
     $\partial_c \partial_{\Delta_{r+1}}b=
     \bigl(\bigl((\dots(b,y_{i_1}),\dots),y_{i_l}\bigr),c \bigr)\in S(G)$ in each such word.
     If
     $\Delta_{r+1}=\Delta$ then we obtain
     $\bigl(([u_1],\dots,[u_r])_c,\partial_{c}\partial_{\Delta}b\bigr)=
      ([\tilde{u}],\partial_{\Delta'} b)$ that is negative to the last summand.

     So,
(\ref{sum}) implies
     \begin{equation}\label{simsum}
     \begin{split}
       (&[u], [v])_{[\tilde{u}v]}\\
          &\equiv_{[\tilde{u}v]}  \sum_{i=1}^{r}\
            \sum_{(\Delta_1,\dots, \Delta_{r+1})\in \Omega_{r+1}(\Delta)}
           \bigl((\partial_{\Delta_{1}}[u_1],\dots,\partial_c\partial_{\Delta_{i}}[u_i],\dots,\partial_{\Delta{r}}[u_r])_{p(\tilde{u},c)}
           \partial_{\Delta_{r+1}}b\bigr),
     \end{split}
     \end{equation}
     where the sign
     ``$\equiv_{[w]}$'' means that the LHS and the RHS  are equal modulo
     summands in
     $S(G)$ less than
     $[w]$.

     On the other hand,
     \begin{equation}\label{int1}
     \begin{split}
       0
      & \equiv_{[\tilde{u}v]}\partial_{\Delta}  \bigl(([u_1],\dots, ([u_i],c), \dots
       [u_r])_c, b\bigr) \\
      & = \sum_{(\Delta_1,\dots,\Delta_{r+1})\in \Omega_{r+1}(\Delta)}
       \bigl(\partial_{\Delta_{1}}[u_1],\dots, \partial_{\Delta_i} ([u_i],c),\dots,
       \partial_{\Delta_r} [u_r])_{p(\tilde{u},c)},\partial_{\Delta_{r+1}} b \bigr)
     \end{split}
     \end{equation}

     Note that each summand of the LHS is a sum of Lyndon--Shirshov words less than
     $[\tilde{u}v]$ and
     \begin{equation}\label{int2}
       \partial_{\Delta_i} ([u_i],c)  = \sum_{(\Delta_{i,1},\Delta_{i,2})\in \Omega_2(\Delta_i)}
         (\partial_{\Delta_{i,1}}[u_i],\partial_{\Delta_{i,2}}c).
     \end{equation}
     moreover, if
     $\Delta_{i,2}\neq \varnothing$ then each word in the RHS of
(\ref{int2}) is a sum of words containing multipliers of the form
     $(c,y_j)$, for some
     $j$. But each such word is in
     $S(G)$ since
     $c>y_j$ and
     $c$ is adjacent to
     $y_j$ for any
     $j=1,2,\dots ,k$. Consequently,
(\ref{int1}) is followed by

     \begin{equation} \label{smsum}
     \begin{split}
      0
      & \equiv_{[\tilde{u}v]} \partial_{\Delta} \bigl(([u_1],\dots, ([u_i],c), \dots
       [u_r])_c, b\bigr) \\
      & \equiv_{[\tilde{u}v]} \sum_{(\Delta_1,\dots,\Delta_{r+1})\in \Omega_{r+1}(\Delta)}
       \bigl((\partial_{\Delta_{1}}[u_1],\dots, (\partial_{\Delta_i}[u_i],c),\dots,
       \partial_{\Delta_r} [u_r])_{p(\tilde{u},c)},\partial_{\Delta_{r+1}} b \bigr)\\
      & = \sum_{(\Delta_1,\dots,\Delta_{r+1})\in \Omega_{r+1}(\Delta)}
       \bigl((\partial_{\Delta_{1}}[u_1],\dots, \partial_c \partial_{\Delta_i}[u_i],\dots,
       \partial_{\Delta_r} [u_r])_{p(\tilde{u},c)},\partial_{\Delta_{r+1}} b \bigr),
     \end{split}
     \end{equation}

     Substituting
(\ref{smsum}) to
 (\ref{simsum}) for
      $i=1,2,\dots, r$ completes the proof of the lemma.
   \end{proof}

   Let us consider more general case.
   \begin{llll}\label{gen}
     Let
     $v=by_1y_2\dots y_k c$, where
     $b>c>y_k\geqslant \dots \geqslant y_1$, and
     $[u]\in S(G)$ is such that
     $[u]=([\tilde{u}],b)$. Then the composition of
     $[u]$ and
     $[v]$ is trivial relative to
     $S(G)$.
   \end{llll}
   \begin{proof}
     Let
     $G=\langle X,E\rangle$ be an arbitrary graph.
     $n$-\emph{coping vertex}
     $x \in X$ is the process of adding vertices
     $x_1,x_2,\dots x_{n}$ such that all of them are adjacent to each other, to the vertex
     $x$, and to the same vertices as
     $x$. In particular,
     $0$-coping any vertex of a graph gives the graph itself. Let
     us order the vertices of the obtained graph as follows:
     the order on
     $X$ is same and
     $x_i>a$ for all
     $i$ if and only if
     $x>a$ in the initial graph.

     For an arbitrary Lyndon--Shirshov word
     $[v]$ and an arbitrary letter
     $x$, let us denote by
     $X([v],x)$ the set of all letters of
     $[v]$, each of which is not less than
     $x$ and has more than one occurrence in it.

     Consider the set
     $X([u],y_1)$ and for each vertex
     $x$ in this set let us perform
     $m(x)$-coping, where
     $m(x)$ is the number of occurrences of
     $x$ in
     $[u]$. It is obvious that the graph obtained after all such
     $m(x)$-copings does not depend on the order of making them. Let us denote this
     graph by
     $G_+=\langle X_+,E_+\rangle$.

     Let
     $x$ be the greatest letter of
     $X([u],y_1)$.  We denote by
     $w'$ the word obtained from
     $w$ by replacing all occurrences of the letter
     $x$ by different letters
     $x_i$.

     Let us show that we can order the letters
     $x_i$ of
     $[u']$ in such a way that
     $[u']$ is a Lyndon--Shirshov word. We proceed by induction
     $\ell(u)$. If the length of
     $[u]$ is equal to
     $1$ or
     $2$ then the statement is obvious because such word cannot contain equal letters.
     Consider an arbitrary Lyndon--Shirshov word
     $[u]$ of the length greater than
     $2$. We have
     $[u]=([w_1],[w_2])$. Since the lengthes of
     $[w_1]$ and
     $[w_2]$ are less than the length of
     $[u]$, we can replace the occurrences of
     $x$ in these words by different letters
     $x_i$ and order these letters in both words
     $[w'_1]$ and
     $[w'_2]$ in such a way that either of them is a Lyndon-Shirshov word.

     If
     $\ell(w_1)=1$ (and so,
     $\ell(w'_1)=1$), then either
     $[w_1]>x$ or
     $[w'_1]=x_s$. In the first case, we are done, because all the
     letters
     $x_i$ are in
     $[w'_2]$. In the second case, let us suppose
     $x_s$ to be greater than other
     $x_i$'s (all other letters are in
     $w'_2$).

     If
     $\ell(w_1)>1$, then
     $[w_1]=([w_{1,1}],[w_{1,2}])$. There are two cases.
     If the first letter of
     $[w_{1,2}]$ or the first letter of
     $[w_2]$ is not
     $x$ then let us suppose that each letter
     $x_i$ of
     $[w'_1]$ is greater than all letters
     $x_i$ of
     $[w'_2]$. If the first letters of both
     $[w_{1,1}]$ and
     $[w_2]$ are equal to
     $x$, then let us suppose that the first letter of
     $[w'_2]$ is greater than the first letter of
     $[w'_{1,1}]$, but less than other letters
     $x_i$ of
     $[w'_{1}]$, and all other letters
     $x_i$ of
     $[w'_2]$ are less than all letters
     $x_i$ of
     $w'_1$.

     So, we obtain a required order on the letters
     $x_i$. Since
     $x$ is the greatest letter having more than one occurrence to
     $[u]$, there are no possibilities except ones considered above.

     Let us repeat this procedure for
     $[u']$ and the greatest letter of
     $X([u'],y_1)=X([u],y_1)\backslash \{x\}$ and so on. Finally, we obtain the word
     $[\check{u}]\in\widehat{S}(G_+,v)$. The same meaning we give to the notatiton
     $[\check{u}_i]$.

     Moreover, for each
     $x$, let us suppose that
     $x>x_j$ for all
     $j$. Then the
     $x$-pattern of
     $[u]$ coincides with the
     $x$-pattern of
     $[\check{u}]$ for any
     $x$ such that
     $x=y_i$ or
     $x=c$.

     Applying Lemma~%
\ref{main} to
     $[\check{u}]$,
     we obtain
     \begin{equation}
     \begin{split}\label{sum2}
       ([\check{u}],&[v])_{[\tilde{\check{u}}v]}  =
            \partial_{\Delta'} [\check{u}]-([\tilde{\check{u}}],\partial_{\Delta'} b)\\
           & = \sum_{(\Delta_1,\dots, \Delta_{r+1})\in \Omega_{r+1}(\Delta)}
           \hspace{-12mm}\bigl((\partial_{\Delta_1}[\check{u}_1],\dots,
           \partial_{\Delta_r}[\check{u}_r])_{p(\tilde{\check{u}},c)},\partial_c
           \partial_{\Delta_{r+1}}b\bigr)\\
     \end{split}
     \end{equation}
     \begin{equation*}
     \begin{split}
           &-([\tilde{\check{u}}],[v])+ \sum_{i=1}^{r}\
             \partial_{\Delta}\bigl(([\check{u}_1],\dots,[\check{u}_ic],\dots,[\check{u}_r])_{p(\tilde{\check{u}},c)},
           b\bigr)\\
           &-\sum_{\stackrel{(\Delta_1,\dots, \Delta_{i},\Gamma,\Delta_{i+1},\dots\Delta_{r+1})\in \Omega_{r+2}(\Delta)}
                  {\Gamma\neq \varnothing}}
            \hspace{-18mm}\bigl((\partial_{\Delta_1}[\check{u}_1],\dots,
            (\partial_{\Delta_i}[\check{u}_i],\partial_{\Gamma}c),\dots
            \partial_{\Delta_r}[\check{u}_r])_{p(\tilde{\check{u}},c)},\partial_{\Delta_{r+1}}b\bigr).
     \end{split}
     \end{equation*}
     From this equality, we can easily see that
     $([\check{u}],[v])_{[\tilde{\check{u}}v]}\equiv_{[\tilde{\check{u}}v]}0$. In particular,
     $([\check{u}],[v])_{[\tilde{\check{u}}v]}=0$ in
     $\mathcal{L}(G_+)$. Therefore, by construction of
     $G_+$, we obtain
     $([u],[v])_{[\tilde{u}v]}=0$ в
     $\mathcal{L}(G)$.
     We are left to prove that the words of the form
     $$
     \begin{array}{rl}
        \bigl((\partial_{\Delta_1}[u_1],\dots,
         \partial_{\Delta_r}[u_r])_{p(\tilde{u},c)},\partial_c
         \partial_{\Delta_{r+1}}b\bigr), & \text{if }\Delta_{r+1}\neq \Delta';\\
         \partial_{\Delta}\bigl(([u_1],\dots,[u_ic],\dots,[u_r])_{p(\tilde{u},c)},b\bigr);&\\
         \bigl((\partial_{\Delta_1}[u_1],\dots,(\partial_{\Delta_i}[u_i],\partial_{\Gamma}c),\dots
            \partial_{\Delta_r}[u_r])_{p(\tilde{u},c)},\partial_{\Delta_{r+1}}b\bigr),
            & \text{if }\Gamma\neq \varnothing
     \end{array}
     $$
     can be represented as sums of words each of which contains a word in
     $S(G)$ as a multiplier, and, moreover, representing these words as linear combinations
     of Lyndon--Shirshov words gives the words less than
     $[\tilde{u}v]$. Since
     $b$ is the second greatest letter of
     $[u]$, each summand of the first kind can be written as a linear combination of
     Lyndon--Shirshov words of the form
     $([w],\partial_c \partial_{\Delta_{r+1}}b)$, where
     $\partial_c \partial_{\Delta_{r+1}}b \in S(G)$. Analogously,
     each summand of the third kind can be written as a linear combination of Lyndon--Shirshov words each of which
     contains a multiplier of the form
     $\partial_{\Gamma}c$. Each such multiplier contains a
     product of the form
     $(c,y_i)\in S(G)$ because all words
     $[u_i]$ are
     $c$-indecomposable. Since
     $c$ is a letter, the greatest letter of each
     $[u_i]$ is greater than
     $c$ while the second greatest letter of it is not greater than
     $c$. Finally, for the summand of the second form representing each word
     $\bigl(([u_1],\dots,[u_ic],\dots,[u_r])_{p(\tilde{u},c)},b\bigr)$
     as a sum of Lyndon--Shirshov words, we obtain the sum of the
     terms of the form
     $([w],b)$, where
     $w$ is a Lyndon--Shirshov word each letter of which has the same number of occurrences in it as in
     $([u_1],\dots,[u_ic],\dots,[u_r])_{p(\tilde{u},c)}$. Consequently,
     all words of this form are in
     $S(G)$, since the vertex corresponding to
     $b$ is adjacent to the vertices corresponding to the letters of
     $u_i$ (%
     $i=1,2,\dots, r$) and to the vertex corresponding to
     $c$.

     By Lemma~%
\ref{less}, for the words of the first kind, all terms in the
     corresponding sum of Lyndon--Shirshov words are not greater
     than
     $[\tilde{u}v]$. Moreover, the greatest word of such sum is equal to
     $([\tilde{u}],[v])$. We can see that for each such summand there is its negative in the sum
     (note that we consider only some summands in the decomposition of
     $\bigl(((\dots([u],y_1),\dots,y_r),c\bigr)$, namely, only those of them, which are in the decomposition of the corresponding word
     $\bigl((\partial_{\Delta_1}[u_1],\dots,
         \partial_{\Delta_r}[u_r])_{p(\tilde{u},c)},\partial_c
         \partial_{\Delta_{r+1}}b)\bigr)$%
     $\bigr)$. Each word of the second and the third kinds has a decomposition to the sum of words not greater than
     $[u_1,\dots,u_ic,\dots,u_rv]$. Consequently, they are certainly less than
     $[\tilde{u}v]$.
   \end{proof}

   Let
   $[v]\in S(G)$ that is
   $[v]=([\tilde{v}],c)$, where
   $c\in X$. Let us show that this word can be represented in the form
   \begin{equation} \label{decom}
     [v]=\bigl((\dots (b,[v_1]),\dots),[v_k]),c\bigr),
   \end{equation}
   where
   $[v_i]\in LS(X)$ for
   $i=1,2\dots,k$,
   $b$ is a letter such that
   $b> c \geqslant v_k\geqslant \dots \geqslant v_1$.
   Let us proceed by induction on
   $\ell(\tilde{v})$. If
   $\tilde{v}=b$ then
   $k=0$ and the assertion follows. Let
   $\ell(\tilde{v})>1$. Then
   $[\tilde{v}]=([w_1],[w_2])$ and
   $b>w_2$ since
   $[v] \in LS(X)$. By the induction hypothesis, there is the
   decomposition
   $[w_1]=(\dots (b,[v_1]),\dots),[v_{k-1}])$, such that
   $b>v_{k-1}\geqslant \dots \geqslant v_1$. Suppose that
   $v_k=w_2$. Since
   $[v]\in LS(X)$, we have
   $b>c>v_k\geqslant v_{k-1}$, which is the required decomposition.

   Now, we can go on to the most general case.

   \begin{llll}\label{algen}
     Let
     $[u],[v]\in S(G)$ and let the last letter of
     $[u]$ be equal to the first letter of
     $[v]$. Then the composition of
     $[u]$ and
     $[v]$ is trivial relative to
     $S(G)$.
   \end{llll}
   \begin{proof}
     By Lemma~%
\ref{gen}, without loss of generality, we can suppose that all the
     letters of
     $\tilde{u}$, which are greater than the greatest letter of
     $v_1$, have no more than one occurrence to
     $\tilde{u}$. In this case,
     $\partial_{\Gamma}[u]$ is a sum of Lyndon--Shirshov words,
     where
     $\Gamma=\{[v_1],\dots,[v_k],c\}$. It is so because
     $[u]=([u_1],\dots, [u_r],b)_{v_1}$ and all the greatest letters
     of the words
     $[u_i]$ are distinct and are not equal to the greatest letter of
     $[v_1]$. Consider the graph
     $G'=\langle X',E'\rangle$, such that
     $X'=X\cup\{y_1,\dots, y_k\}$, where
     $X\cap \{y_1,\dots, y_k\}=\varnothing$ and
     $E'$ is obtained from
     $E$ by adding to it all edges
     $\{x,y_i\}$, where
     $x\in X$ and
     $x$ is adjacent to all vertices such that the corresponding letters are in
     $u_i$. It is obvious that, in this case,
     $c$ is adjacent to the vertices
     $y_1,\dots,y_k$. Consequently, Lemma~%
\ref{main} implies the composition of
     $u$ and
     $\bigl(((\dots(b,y_1),\dots),y_k),c\bigl)$ is trivial
     relative to
     $S(G')$. It can be seen from the proof of this lemma that for
     the new set of vertices, the composition of
     $[u]$ and
     $[v]$ can be represents as a sum of the words of the form
     $([u_0], \partial_c \partial_{\Delta_0 }b)$
     or
     $\bigl((\dots((([u_1],\dots, ([u_i],c),\dots,[u_r])_{p(\tilde{u},c)}, b), y_1),\dots), y_k \bigr)$,
     where
     $\Delta=\{y_1,\dots,y_k\}$ and
     $\Delta_0$ is a proper subset of
     $\Delta$. It is clear that replacing
     $y_i$ by corresponding
     $[v_i]$ in these words, we obtain the words, containing subwords in
     $S(G)$ less than
     $[\tilde{u}v]$. Consequently,
     $([u],[v])_{[\tilde{u}v]}\equiv_{[\tilde{u}v]} 0$ and we are done.
   \end{proof}

   Let
   $[v]$ be a Lyndon---Shirshov word. We say that a Lyndon---Shirshov word
   $[u]$ is \emph{a subword} of
   $[v]$, if
   $u$ is a subword of
   $v$.

   Now, we can formulate the main results of this paper.

   \begin{ttt}\label{result}
     Let
     $G=\langle X, E\rangle$ be an undirected graph. Then a linear basis of
     $\mathcal{L}(G)$ consists of all non-associative Lyndon--Shirshov
     not containing subwords of the form
     $[u]=([\tilde{u}],b)$, such that the greatest letter of
     $u$ has exactly one occurrence in it,
     $b$ is the second largest letter of
     $u$ and
     $G$ contains all edges of the form
     $\{b,x\}$, where
     $x$ is a letter of
     $\tilde{u}$.
   \end{ttt}

   \begin{proof}
     By Lemma~%
\ref{id}, the set
     $S(G)$ (consisting of the words of the form
     $[u]=([\tilde{u}],b)$, such that the greatest letter of
     $[u]$ has exactly one occurrence in it,
     $b$ is the second greatest letter of it and all vertices
     corresponding to the letters of
     $[\tilde{u}]$ are adjacent to the vertex corresponding to
     $b$) is a subset of the ideal generated by the subset of the ideal generated by the words
     $(x,y)$, where
     $\{x,y\}\in E$.

     On the other hand, since
     $(x,y)\in S(G)$, for
     $\{x,y\}\in E$, we obtain
     $I(G)\subseteq \mathrm{Id}(S(G))$.

     By Lemma~%
\ref{main}, Lemma~%
 \ref{gen}, and Lemma~%
 \ref{algen},
     $S(G)$ is complete under compositions. So the theorem holds.
   \end{proof}

   It clearly follows from Theorem~%
\ref{result} that the equality problem is algorithmically solvable
   for the partially commutative Lie algebras. Indeed, any
   non-associative word can be represented as a linear combination
   of Lyndon---Shirshov words. If a summand of this linear
   combination (denote it by
   $\alpha [v]$, where
   $\alpha\in R$) contains a subword
   $[u]$, then by computing corresponding compositions of inclusion (see Definition~%
\ref{inclus}) we can represent
   $[v]$ as a linear combination of words not containing elements in
   $S(G)$ as subwords.

   Let
   $G=\langle X, E\rangle$ be an undirected graph. Denote by
   $\mathcal{L}(G,n)$ the partially commutative nilpotent Lie
   algebra with the level of nilpotence
   $n$ corresponding to
   $G$, i.e. the algebra
   $\mathcal{L}(G)/I$, where
   $I$ is an ideal consisting of all words of the length not less than
   $n$.

   \begin{ttt}
     Let
     $G=\langle X, E\rangle$ be an undirected graph. Then a linear basis of
     $\mathcal{L}(G,n)$ consists of all non-associative
     Lyndon--Shirshov words of the length not greater than
     $n-1$ not containing subwords of the form
     $[u]=([\tilde{u}],b)$, such that the greatest letter of
     $u$ has exactly one occurrence in it,
     $b$ is the second largest letter of
     $u$ and
     $G$ contains all edges of the form
     $\{b,x\}$, where
     $x$ is a letter of
     $\tilde{u}$.
   \end{ttt}

   \begin{proof}
     The proof of this theorem is analogous to the proof of the last one. We only have to note that
     $\mathcal{L}(G,n)=\mathrm{Lie}_{R}(X) / \mathrm{I}(G,n)$, where
     $\mathrm{I}(G,n)$ is the ideal of the free Lie algebra generated by
     $S(G)$ together with the set of Lyndon--Shirshov words of the length less than
     $n$. By Lemma~%
\ref{main}, Lemma~%
 \ref{gen}, and Lemma~%
 \ref{algen}, this set is complete under composition. So, we are done.
   \end{proof}

   \noindent
   {\bfseries Acknowledgments:} The author is very grateful to Prof.~E.\,I.\,Timoshenko for very helpful
   criticism and suggestions about this paper.


\begin{thebibliography}{99}
      \bibitem{Bo72}
       Bokut, L.\,A. Unsolvability of the Equality Problem and Subalgebras of
       Finitely Presented Lie Algebras, Math USSR Izvestia, {\bfseries 6}, 1972,
       1153-1199 (in Russian).
      \bibitem{BFKK00} Bokut, L.\,A., Fong, Y., Ke, V.-F.,
       Kolesnikov, P.\,S., Gr\"{o}bner--Shirshov bases in algebra
       and conformal algebras, Fundamental and applied
       mathematics, {\bfseries 6(3)}, 2000, 679--716.
      \bibitem{BC08} Bokut, L.\,A., Chen, Y.,
       Gr\"{o}bner--Shirshov Bases for Lie Algebras: After
       A.\,I.\,Shirshov, preprint, {\ttfamily
       arXiv:math.RA/08041254v1}.
      \bibitem{BS01} Bokut, L.\,A., Shiao L.-S., Gr\"obner--Shirshov bases for
       Coxeter groups, Comm. Algebra, {\bfseries 29}, 2001, no. 9, 4305-4319.
      \bibitem{GT09}
       Gupta C.\,K., Timoshenko E.\,I., Partially Commutative Metabelian
       Groups: Centralizers and Elementary equivalence, Algebra i
       Logika, {\bfseries 48}, 3, 2009, 309--341 (in Russian).
      \bibitem{DK92'}
       Duchamp G., Krob D., The Free Partially Commutative Lie Algebra: Bases and Ranks,
       Advances in Mathematics, {\bfseries 92}, 1992, 95--126.
      \bibitem{DK92}
       Duchamp G., Krob D., The lower central ceries of the free partially commutative
       group, Semigroup Forum, {\bfseries 45}, 1992, 385-–394.
      \bibitem{DK93}
       Duchamp G., Krob D., Free Partially Commutative Structures,
       Journal of Algebra, {\bfseries 156}, 1993, 318--361.
      \bibitem{DKR07}
       Duncan A.\,J., Kazachkov I.\,V., Remeslennikov V.\,N.
       Parabolic and quasiparabolic subgroups of free partially commutative
       groups, Journal of Algebra, {\bfseries 318}, 2, 2007, 918--932.
      \bibitem{KMNR80}
       Kim, K.\,H., Makar-Limanov, L., Neggers, J., Roush, F.\,W.,
       Graph algebras, Journal of Algebra, {\bfseries 64}, 1980, 46--51.
      \bibitem{Ser89}
       Servatius H., Automorphisms of graph groups, J. Algebra, {\bfseries 126},
       1, 1989
       34-–60.\bibitem{She05}
       Shestakov S.\,L., The Equation
       $[x, y] = g$ in Partially Commutative groups,
       Siberian Math. J., {\bfseries 46}, 2, 2005, 466–-477 (in Russian).
      \bibitem{She06}
       Shestakov S.\,L., The Equation
       $x^2 y^2 = g$ in Partially Commutative groups,
       Siberian Math. J., {\bfseries 47}, 2, 2006, 463–-472 (in Russian).
      \bibitem{sh58}
       Shirshov, A.\,I. On Free Lie Rings, Math. sb, {\bfseries
       45(87)}, 1958, 113--122 (in Russian).
      \bibitem{sh62}
       Shirshov, A.\,I. Some Algorithmic Problems for Lie algebras, Siberian Math. J.,
       {\bfseries 3}, 1962, 292--296.
      \bibitem{sh62'}
       Shirshov, A.\,I. On Bases of Free Lie Algebra, Siberian Math. J., {\bfseries 3},
       1962, 297--301.
      \bibitem{T10} Timoshenko E.\,I., Eniversal Equivalence of
      Partially Commutative Metabelian Groups, Algebra i Logika,
       {\bfseries 49}, 2 (2010), 263--290 (in Russian).
     \end{thebibliography}
 \end{document}